\documentclass[12pt]{amsart}
\usepackage{amsmath,amsthm, amssymb}
\usepackage{subcaption}
\usepackage{graphicx} \graphicspath{{figs/}}
\usepackage{hyperref}
\usepackage{enumerate}

\newcommand{\ben}{\begin{enumerate}} 
\newcommand{\een}{\end{enumerate}}

\def\bea{\begin{eqnarray*}}
  \def \eea{\end{eqnarray*}}

\def\H{\mathbb{H}^2}
\def\dH{\partial\H}

\def\MS{M}
\def\Cov{M_3}

\def\Z{\mathbb{Z}} 
\def\FF{\mathbb{F}} 
\def\PP{\mathbb{P}} 
\def\SS{\mathbb{S}} 

\newcommand\PSL{\mathrm{PSL}}
\DeclareMathOperator\Stab{Stab}
\DeclareMathOperator\Vol{Vol}
\DeclareMathOperator\Tr{Tr}

\newcommand\abs[1]{\left| {#1} \right|}
\newcommand\isom{\simeq}

\newtheorem{thm}{Theorem}[section]
 
\newtheorem{lem}[thm]{Lemma}
\newtheorem{prop}[thm]{Proposition}

\theoremstyle{definition}

\theoremstyle{remark}
\newtheorem*{rem}{Remark}

\newtheorem{example}[thm]{Example}

\title[An upper bound]{An upper bound for the volumes of complements of periodic  geodesics}
\author{Maxime Bergeron}
\address{The University of British Columbia, Vancouver BC, Canada}
\email{mbergeron@math.ubc.ca}
\urladdr{http://www.math.ubc.ca/~mbergeron}
\thanks{M.B.\ was supported by an NSERC Alexander Graham Bell CGS-D Scholarship}

\author{Tali Pinsky}
\email{tali@math.tifr.res.in}
\urladdr{http://www.math.tifr.res.in/~tali/}

\author{Lior Silberman}
\email{lior@math.ubc.ca}
\urladdr{http://www.math.ubc.ca/~lior/}
\thanks{L.S.\ was partly supported by an NSERC Discovery Grant}

\date{\today}

\begin{document}

\begin{abstract}
A periodic geodesic on a surface has a natural lift to the unit tangent
bundle; when the complement of this lift is hyperbolic, its volume
typically grows as the geodesic gets longer.  We give an upper bound
for this volume which is linear in the geometric length of the geodesic.
\end{abstract}

\maketitle

\section{Introduction}
A closed curve on a surface $S$ can be naturally lifted to the unit tangent
bundle $T^1S$ by traversing the curve in a chosen direction and
associating to each point on the curve its tangent direction (note that
there are two lifts, corresponding to the two directions along the curve).
Such a lift is an embedding of $\SS^1$ into the $3\,$-manifold $T^1S$
and, considering its ambient isotopy class, we obtain
a knot.

If we equip $S$ with a hyperbolic Riemannian metric, the isotopy class of each
(non-peripheral and non-trivial) closed curve contains a periodic geodesic
(the representative of shortest length). 
The associated knot is then a closed orbit of the geodesic flow on $T^1S$ which does not depend on the chosen metric.
Such periodic geodesics have long been an object of interest  from the topological \cite{Fried:TransitiveAnosov,Thurston:GeomDynSurfaceDiff}
and dynamical \cite{Anosov:GeodesicFlows} points of view. However, except for a few promising results for closed geodesics on the modular surface \cite{dehornoy2015geodesic,Ghys:KnotsDynamics,kelmer2012quadratic, sarnak2010linking} and other triangle groups \cite{dehornoy2015geodesic}, they have essentially never been studied as knots.
%One of the reasons for this is that the main tool for the previous result as well as for ours is
%\'Etienne Ghys's proof  \cite{Ghys:KnotsDynamics}, from 2006, that modular geodesics are
%embedded on a certain branched surface called the \emph{template} in a
%way which preserves all its topological properties (see
%Section~\ref{sec:coding} for the details).
One of the reasons for this is that these knots are embedded in $T^1S$ which, unless $S$  happens to be the modular surface,
cannot be embedded into $\mathbb{S}^3$. As such, there is a  limited supply of knot theoretic invariants available to try and characterize them.

Nevertheless,  the complement of a lift of a periodic geodesic in the unit tangent bundle
is a hyperbolic $3\,$-manifold as soon as the curve is filling
\cite{FoulonHasselblatt:ContactAnosov} in the surface. In this case,
its volume is a topological invariant and we give the first known estimate for it:
\begin{thm}\label{thm:mainthm}
Let $S = \H/\Gamma$ be a hyperbolic surface. There is a constant $C_S>0$
such that, for any finite set $\gamma$ of (primitive) periodic geodesics
on $S$, we have
\begin{equation}\label{eq:mainbound}
\Vol(T^1S\setminus \gamma) \leq C_S \abs{\gamma}.
\end{equation}
Here, $\abs{\gamma}$ denotes the total length of the geodesics in $\gamma$.
\end{thm}

\begin{rem}
The length $\abs{\gamma}$ is measured with respect to the fixed hyperbolic
metric on $S$.  We emphasize, however, that the volume is topological as is
the linearity of the bound.  Changing the metric will only affect the constant.
\end{rem}

Our proof of Theorem~\ref{thm:mainthm} begins with the modular
surface $\MS$. In this case, there is a  symbolic description of the closed geodesics: they can be coded by their  continued fraction expansion  
 as  done by E.\ Artin \cite{Artin:SymbolicCoding,Series:ModularSurfaceContinuedFractions}. On the other hand, this symbolic description was recently shown by Ghys \cite{Ghys:KnotsDynamics} to have a topological interpretation.
The key to our linear bound for $\MS$ lies in the connection between a 
decomposition of $T^1\MS$ %into basic building blocks
and a complexity measure for the codes of the geodesics in $\gamma$ obtained via Ghys's explicit   \emph{template}  (c.f.
Section~\ref{sec:coding}).

\begin{rem}
In Example~\ref{exa:bounded-volume}
we exhibit a family of arbitrarily long geodesics on the modular surface
whose associated knot complements have uniformly bounded volume.
Nevertheless, numerical evidence
\cite{BrandtsPinskySilberman:KnotVolumesNumerics_preprint}
indicates that, in some situations, the growth is  linear in the
geometric length of the geodesics.
\end{rem}

\begin{rem}
It is interesting to compare our result for the modular surface with Gu{\'e}ritaud and Futer's
volume estimates for once-punctured torus bundles and two-bridge link
complements \cite{Gueritaud:CanonicalTriangulations}.  They use the same symbolic coding 
 for elements in $SL_2(\Z)$.   However, their volume bounds
are linear in the period of the continued fraction expansion, whereas ours
are linear in the geometric length (c.f. Bridgeman \cite{bridgeman1998bounds}).
\end{rem}

The paper unfolds as follows:
In Section~\ref{sec:coding}, we review the coding of geodesics
on the modular surface $\MS$ by positive words.  In Section~\ref{sec:combupperbound}, we obtain
a volume bound in terms of the combinatorics of the coding, from which we deduce
Theorem~\ref{thm:mainthm} for $\MS$. %by bounding the
%combinatorial complexity in terms of the geometric length.
Finally, in Section~\ref{sec:upperbound}, we deduce the general case of  Theorem~\ref{thm:mainthm}
by relating covers of punctured surfaces to covers of $\MS$.

We are grateful to Juan Souto for suggesting that we study this invariant
and for offering useful insights.  We would also like to thank Yair Minsky, Jessica Purcell
and Kasra Rafi for helpful discussions.
%%%%%%%%%%%%%%%%%%%%%%%%%%%%%%%%%%%%%%%%%%%%%%%%%%%%%%%%%%%%%%%%%%%%%%%%%%%%%
\section{Background: Coding of geodesics on the modular surface}
\label{sec:coding}
In this section, we review Ghys's construction \cite{Ghys:KnotsDynamics} of a template 
 and the associated coding for the geodesic flow
on the modular surface.

We start with the isomorphism $\PSL_2(\Z)\isom C_2*C_3$.
For later reference, we fix elements $\kappa_0$ and $\omega\in\PSL_2(\Z)$ of
order $2$ and $3$, respectively, so that
$\PSL_2(\Z)=\left<\kappa_0\right>*\left<\omega\right>$.
Restricting every homomorphism to the generating set, we see that
$\PSL_2(\Z)$ has an essentially unique surjection onto $C_3$, and hence a
unique normal subgroup $\Gamma_3$ of index $3$.
We write $\Cov := \H/\Gamma_3$ for the associated hyperbolic manifold,
the unique normal three-fold cover of $\MS$, depicted in Figure \ref{fig:M3}.

\begin{figure}[ht]
\includegraphics[width=7cm]{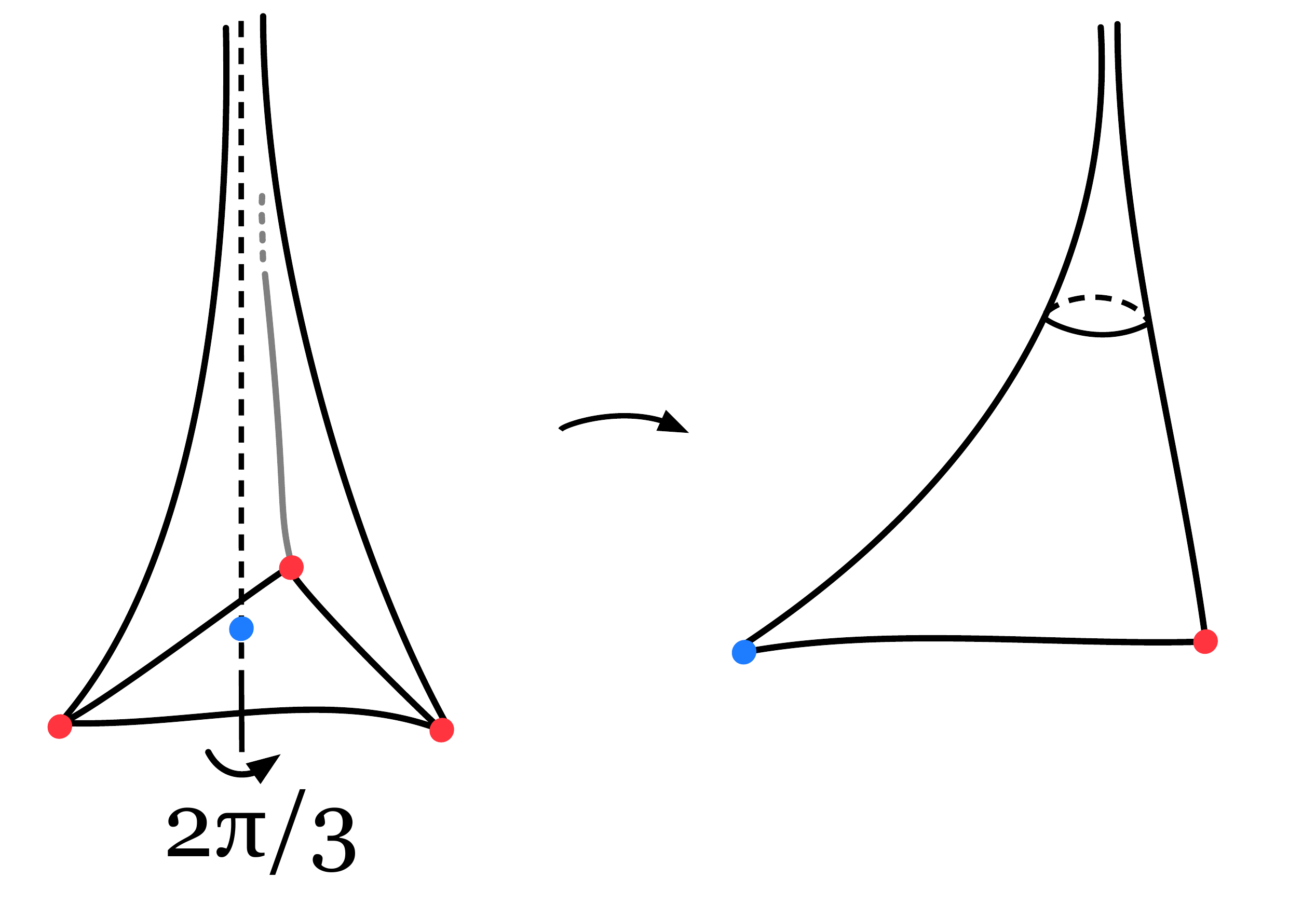}
\caption{The three-fold cover $\Cov$ of the modular surface $\MS$.}\label{fig:M3}
\end{figure}

\begin{rem}
For an alternative construction of $\Gamma_3$, note that
$\PSL_2(\FF_3) \isom S_4$ (consider the action on $\PP^1(\FF_3)$)
and that the $2$-Sylow subgroup of $S_4$ is normal of 
index $3$; this gives a surjection onto $C_3$ and shows that the kernel
contains $\Gamma(3)$ so that $\Gamma_3$ is a congruence subgroup.
\end{rem}

Let $p\in \H$ be the point fixed by $\omega$ and notice that, since 
$\Stab_{\PSL_2(\Z)}(p)$ coincides with $\left<\omega\right>$,  $p$
is not a fixed point of an elliptic element of $\Gamma_3$.
It follows that the associated Dirichlet domain $U_3$ (the set of points
of $\H$ closer to $p$ than to any other point of the orbit $\Gamma_3\cdot p$)
is a fundamental domain for $\Cov$.
In fact, as shown in Figure \ref{fig:3cover}, $U_3$ is an ideal triangle.
Moreover, $\kappa_0$ fixes a point $q_0$ along one of the sides of this
triangle which we
label $J_0$, and its conjugates $\kappa_i = \omega^i \kappa_0 \omega^{-i}$
fix points $q_i = \omega^i q_0$ along the other two sides $J_1$ and $J_2$.
For convenience (as shown in the figure) we choose our identification of $\H$
with the disc so that $p$ is the centre of the disc.  In that case $\omega$
acts by rotation by the angle $2\pi/3$, cyclically permuting
the vertices of the triangle, the arcs $J_i$ connecting them, and the
elliptic fixed points $q_0,q_1$ and $q_2$.

We now use this picture to study certain geodesics on $\MS$.
Observe first that any set $\gamma$ of periodic geodesics lifts to a set
$\tilde\gamma$ of closed curves in $\Cov$ (each periodic geodesic in $\MS$
lifts either to three periodic geodesics of the same length, or to a single
geodesic of three times the length, but this is immaterial for our arguments).
Moreover, any geodesic in $\Cov$ has a lift to an (infinite) geodesic on $\H$
connecting two points on the boundary. This lift may be chosen to cross any
particular fundamental domain for $\PSL_2(\Z)$, specifically $U_3$.
If the geodesic is periodic, its ends cannot lie on
the cusps of $\PSL_2(\Z)$ and, in particular, on the vertices of $U_3$.
Hence, acting by the element of order $3$, we may choose the lift to start at
$I_0 \subset \dH$, so that it will enter the triangle $U_3$ through its
side $J_0$.

\begin{figure}[ht]
\includegraphics[width=7cm]{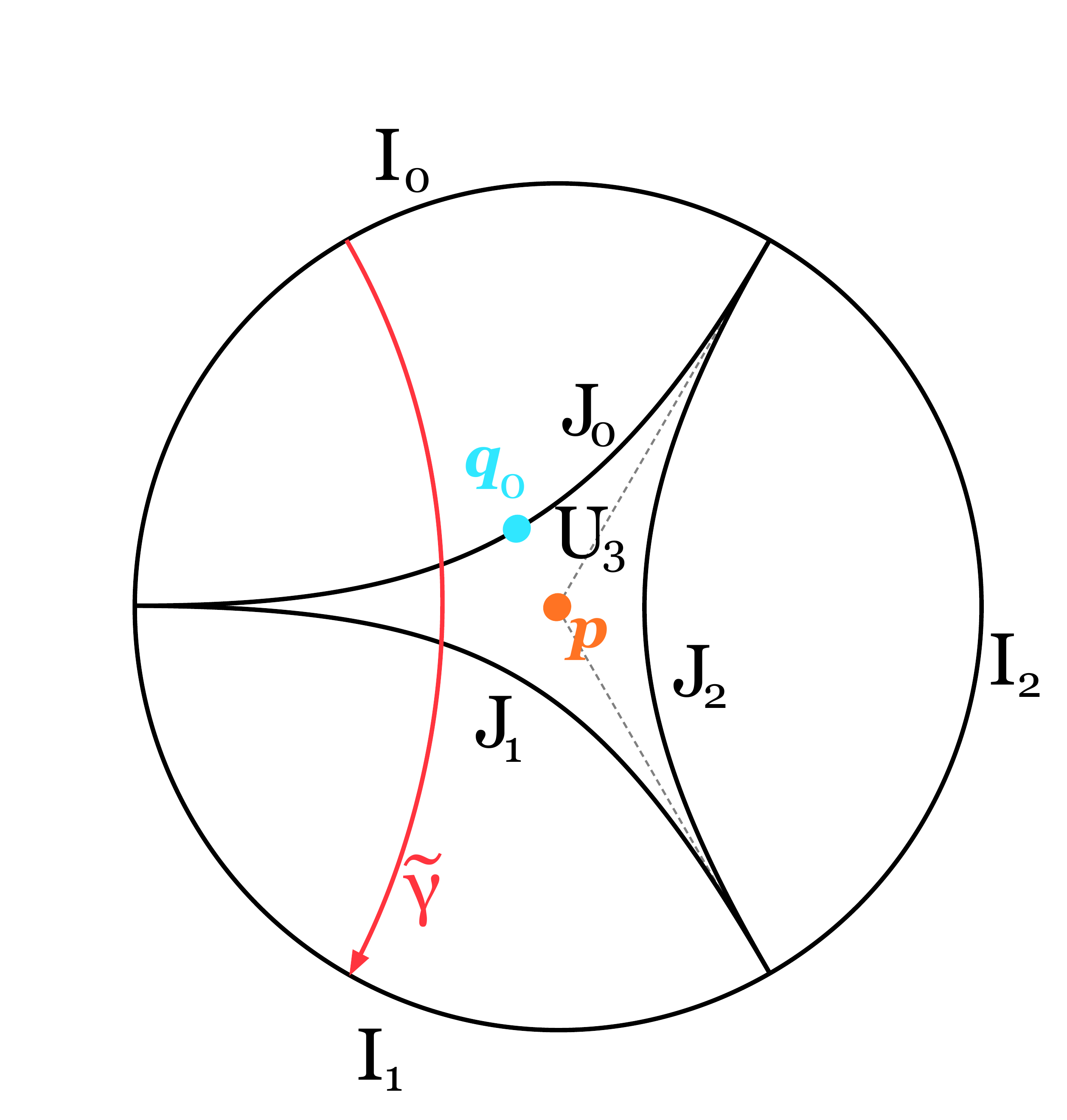}
\caption{A fundamental domain, $U_3$, for a three-fold cover $\Cov$ of the modular surface.}\label{fig:3cover}
\end{figure}

We have seen that it is enough to encode geodesics through $U_3$ starting at
$I_0$.  Accordingly, let $\tilde\gamma$  be such an infinite geodesic.  We now
construct a different (disconnected) lift of $\gamma$, consisting of a sequence of segments in $U_3$.
The first part of this new lift will be the segment of $\tilde\gamma$ between
its entry point to $U_3$ along $J_0$ and its exit point along either $J_1$ or
$J_2$.  We begin our code by $x$ or $y$ depending on the two possibilities.

Suppose our segment ends on $J_i$.  Acting by $\kappa_i$, the remaining
ray of $\tilde\gamma$ (the part after the end of the segment) now
\emph{begins} on $J_i$ (but usually not on the point where the segment
ended). Rotating by a power of $\omega$, we may assume instead that the
remaining ray enters again via $J_0$.  It will then exit via one of the
other sides and once again: we break off a segment, record a letter $x$ or $y$,
apply $\kappa_1$ or $\kappa_2$ and a rotation, and continue as before.

If $\tilde\gamma$ projects to a periodic geodesic $\gamma$ on $\MS$,
the resulting infinite word will be periodic.  In that case the
\emph{code} $w_\gamma$ for $\gamma$ will be the the primitive part of this periodic word.
We write $n_\gamma$ for the number of (cyclic) subwords of the form $xy$
in $w_\gamma$.  This corresponds to the period in the continued fraction
expansion of $\gamma$.  In more generality, for
$\tilde\gamma=\cup_{i=1}^k\tilde\gamma_i$ (a set of geodesics projecting to
a collection of periodic geodesics $\gamma$ on $\MS$) we define $n_\gamma$
to be the sum $n_\gamma=\sum_{i=1}^k n_{\gamma_i}$.
It is shown in \cite{Ghys:KnotsDynamics} that, up to cyclic permutation,
the word $w_\gamma$ only depends on the projection of $\tilde\gamma$ to
$\MS$ and, conversely, that any such word encodes a periodic geodesic.

As noted in the introduction, this coding also arises from the continued
fraction expansion of a primitive modular geodesic (c.f. Series  \cite{Series:ModularSurfaceContinuedFractions}). 
There, $n_{\gamma}$ is exactly half the period of the (even) continued
fraction corresponding to $\gamma$. 
Constructing the code as above, however, yields a natural interpretation in terms of a \emph{template} embedded in $T^1\MS$ whose
existence is due to Birman and Williams \cite{BirmanWilliams:KnottedOrbitsII}
and which was explicitly constructed by Ghys \cite{Ghys:KnotsDynamics}.
This template is a branched surface with boundary, equipped with a
semi-flow such that any finite set $\tilde\gamma$ of geodesics may be
deformed into closed flow orbits on the template by an ambient isotopy
(of $\tilde\gamma$ together with the tangent direction). 
 Ghys \cite{Ghys:KnotsDynamics} showed that the template can be
chosen to have a single branch-line lying over $J_0$ in $T^1\MS$ together
with two bands such that, after the isotopy, one contains the geodesic
segments passing in $U_3$ from $J_0$ to $J_1$ while the other contains
the segments passing from $J_0$ to $J_2$.  From this point of view the
coding given above describes the sequence by which the image of
$\tilde\gamma$ under the isotopy travels through the two template bands.
 Figure \ref{Ghys's template} depicts the template with its
embedding in $T^1\MS$, i.e., the complement of a trefoil knot in $\mathbb{S}^3$.

\begin{figure}[ht]
\includegraphics[width=6cm]{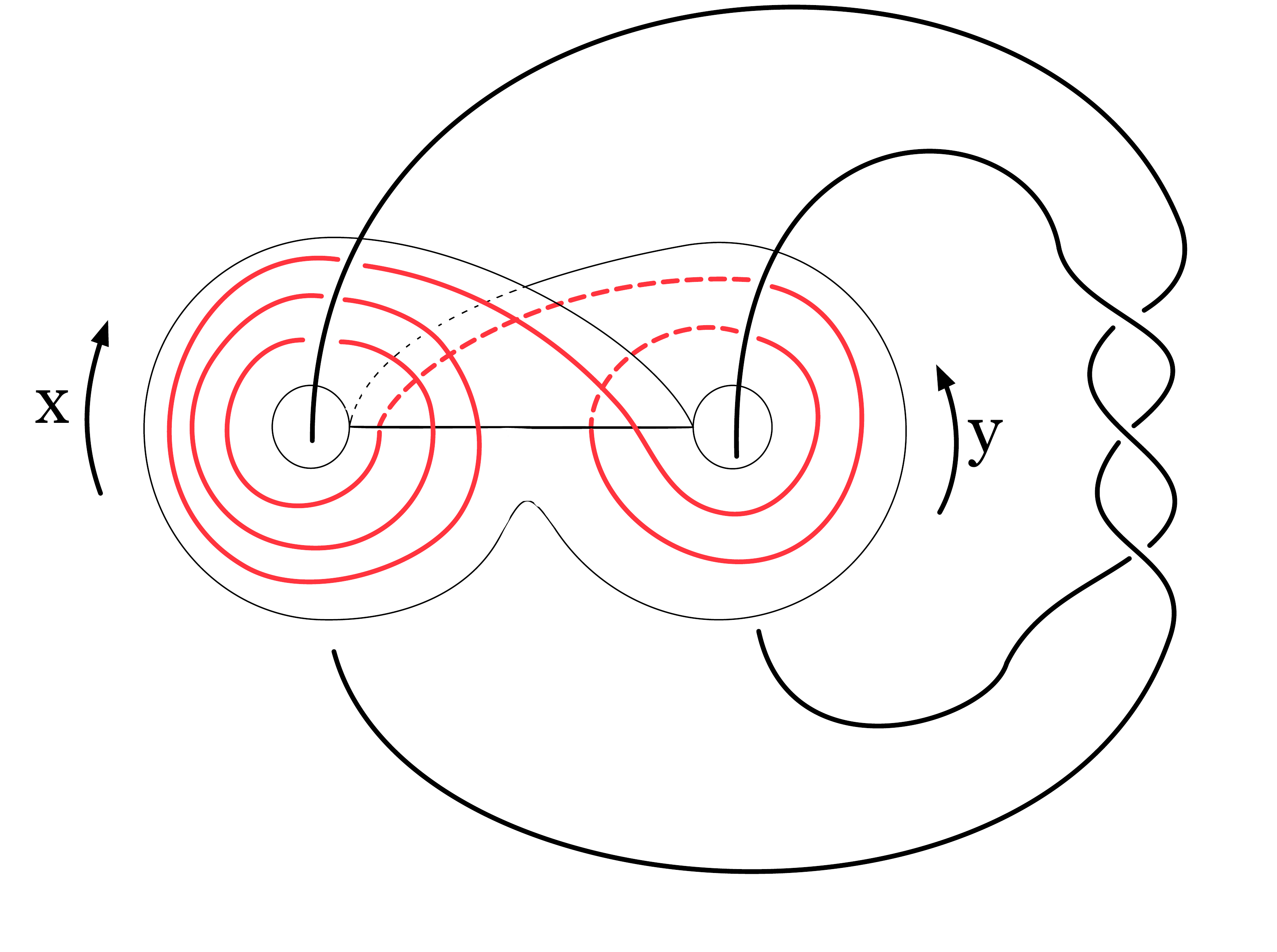}
\caption{Ghys's template for the modular surface embedded in its unit tangent bundle (the complement of a trefoil in $\mathbb{S}^3$) along with a   periodic
orbit corresponding to the code word $x^3y^2$.}\label{Ghys's template}
\end{figure}

Since our first upper bound on the volume of the complement of the lift of $\gamma$ in $T^1\MS$ will be obtained in terms of the combinatorics of $w_\gamma$, we will need the following relation between this codeword and the length of $\gamma$: 
\begin{lem}\label{lem:length}
The length $L$ of the closed geodesic $\gamma$ coded by the word
$x^{k_1}y^{m_1}\dots x^{k_{n_{\gamma}}}y^{m_{n_\gamma}}$
(recall that $k_i,m_i\in \Z_{\geq 1}$) satisfies
\begin{align*}
  L & \geq \sum_{i=1}^{n_\gamma} \left(\log(k_i) + \log(m_i)\right),\text{ and} \\
  L & \geq (\log 2)n_\gamma.
\end{align*}
\end{lem}
\begin{proof}
Letting $x = \begin{pmatrix} 1 & 1 \\ 0 & 1\end{pmatrix}$
and $y = \begin{pmatrix} 1 & 0 \\ 1 & 1\end{pmatrix}$, the geodesic $\gamma$
corresponds to the conjugacy class in $\PSL_2(\Z)$ of the element
$\prod_{i=1}^{n_\gamma} x^{k_i} y^{m_i}$.  We note that
$$x^k y^m = \begin{pmatrix} km+1 & k \\ m & 1\end{pmatrix}\,.$$
Since the entries of a product of matrices are monotone in the entries of
the factors (when the entries are non-negative) we see that 
$$ \gamma = \prod_i x^{k_i} y^{m_i} \geq
  \begin{pmatrix} \prod_i (k_i m_i + 1) & 0 \\ 0 & 1\end{pmatrix}.$$
The entries of this matrix are connected to the length by
$\Tr(\gamma) = e^{L} + e^{-L}$, so that
$e^L \geq \Tr(\gamma) - 1 \geq \prod_i (k_i m_i + 1)$.  The first claim
follows immediately, and for the second we note that $k_i m_i +1 \geq 2$.
\end{proof}

%%%%%%%%%%%%%%%%%%%%%%%%%%%%%%%%%%%%%%%%%%%%%%%%%

\section{Triangulation of link complements and the combinatorial upper bound}
\label{sec:combupperbound}
Let $\gamma$ be a (filling) collection of periodic geodesics on $\MS$, let $w_\gamma$ be their coding and consider their lifts $\tilde{\gamma}$ in $T^1\MS$, isotoped onto Ghys's template as in the previous section. Proceeding to drill out $\tilde\gamma$
from $T^1\MS$, we obtain a hyperbolic $3$-manifold. We will estimate
the volume of this manifold by constructing a triangulation.  However,
the number of tetrahedra in our triangulation is proportional to the word
length of $w_\gamma$, so directly using the boundedness of volumes of
hyperbolic tetrahera does not give a useful bound.
The following example illustrates the problem:

\begin{example}\label{exa:bounded-volume}
Let $\gamma$ be a geodesic on $\MS$ coded by the word $x^ny^m$ for some natural numbers $n$ and $m$.
Its length is roughly proportional to $\log(n+m)$ and, in particular, tends
to infinity with $n$ and $m$.  The number of crossings on the template (which will be proportional  to the number in tetrahedra in our triangulation) is then $n+m$.
On the other hand, these geodesics correspond to a knot winding more and more
around the trefoil, first in one template ear and then in the other.
Thus, their volumes are all bounded by the volume of
$T^1\MS\setminus(\gamma\cup\alpha\cup\beta)$ where $\gamma$ is the geodesic
corresponding to $xy$, and $\alpha$ and $\beta$ are both trivial knots
encircling one strand of $\gamma$ and the strand of the trefoil in the
centre of the corresponding ear (c.f Adams \cite{adams1985thrice}). 
Note that $n_\gamma=1$ for any geodesic $\gamma$ in this family. 
\end{example}

In order to circumvent this difficulty, in Section \ref{subsec:piece-volume-bounds} we will establish that large groups of tetrahedra in our triangulation
share an edge. The volume of a tetrahedron can be expressed in terms of its
dihedral angles via the so-called Lobachevsky function, and we will analyze
it to show that the sum of the volumes of the tetrahedra in the group grows
sub-linearly in the number of tetrahedra. This will allow us to attain
our linear bound for $\MS$.
\begin{rem}
	With the hope that no confusion will arise, we will no longer distinguish a collection $\gamma$ of geodesics on $\MS$ from its lift $\tilde\gamma$ to $T^1\MS$. Both of them will be  henceforth denoted by $\gamma$.
\end{rem}

\subsection{Triangulation}\label{subsec:triangulation}
Let us fix once and for all a finite collection $\gamma$ of periodic geodesics on $\MS$.  In order to triangulate $T^1\MS\setminus \gamma$, we first decompose it  into
pieces of three types ($A$, $B$ and $C$) along with a ``remainder'' ($D$). There will be two pieces of type $A$, two pieces of type $B$, four pieces of type $C$ and one piece of type $D$. 
As such, writing $\Vol(A)$ (respectively $\Vol(B)$, $\Vol(C)$ and $\Vol(D)$) for an upper bound
on the volume of a piece of type $A$ (respectively $B$, $C$ and $D$), the volume of $T^1\MS$ will then be bounded from above by
\begin{equation}\label{decomposition equation}
2\Vol(A)+2\Vol(B)+4\Vol(C)+\Vol(D).
\end{equation}

\subsubsection{Cutting}
Considering Figure \ref{Ghys's template} once again, we define some cutting segments on the template that  will be used to decompose $T^1\MS\setminus\gamma.\,\,$ 
First, follow the flow-lines up from the critical point into both template ears. 
After propagating up past the branch-line, stop at some point $p$ in the $x$
ear and some point $q$ in the $y$ ear.  Then, connect each of these points to
the outer boundary of the template ear by two additional segments as shown in Figure \ref{cutting}.

\begin{figure}[ht]
\includegraphics[width=6cm]{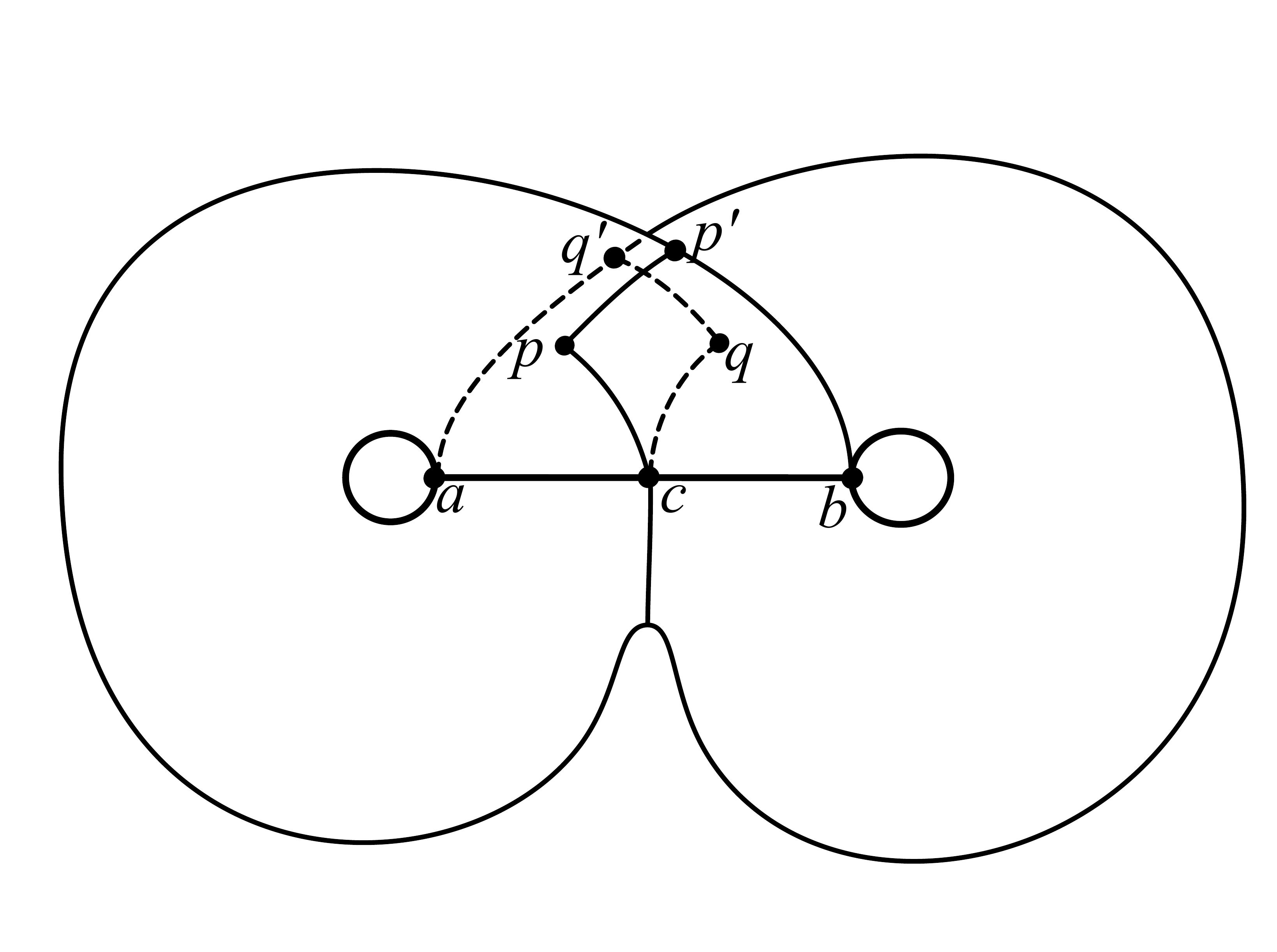}
\caption{The cutting segments on the template; the segments connecting $p$ and $q$ to the critical point $c$ are chosen to lie along the flow-lines.}\label{cutting}
\end{figure}

\subsubsection{Pieces}
We can now describe the four different types of pieces shown in Figure \ref{blocks}.
All relative position adjectives refer to the figure. 

\begin{description}
\item[Type $A$] Choose a point $o_1$ on the trefoil,
above the centre of the $x$ ear, and use it to cone-off the (upper side of the) $x$ ear without including the quadrilateral $cbp'p$. This is the first piece of type $A$, we call it $A_1$.
%We choose the points $o_1$, $a$, $c$, $p$ and $p'$  as its vertices as shown in Figure \ref{blocks}(a) and 
 The second piece of type $A$ is the cone of the bottom side
of the $y$ ear with a point $o_1'$ on the trefoil below the centre of the $y$ ear. We call it $A_2$.

%, the inner boundary of the $x$ band (which is a meridian of the trefoil connecting $a$ to itself) and the outer boundaries of the template as edges (c.f. Figures~\ref{blocks}(a) and \ref{fig:band}).

\item[Type $B$]   Consider the prism
obtained by taking the product of the $y$ ear without the quadrilateral
$qq'ac$ with the segment $cp$ as in  Figure \ref{blocks}(b).
The quadrilateral $cbp'p$ is contained in this cylinder
as $cb\times cp$ and the upper face of this cylinder is bounded by an edge connecting $p$ to itself along the outer side and an edge connecting $p'$ to itself along the inner side. Choose a point $o_2$ on the trefoil
above the centre of the $y$ ear and cone-off the upper face of this cylinder except for the segment $pp'$ which is coned-off to $o_2$ 
by two $pp'o_2$ triangles: one on each sides of the $\tilde\gamma$ strands coming into the $y$-ear. This ``coned-off'' cylinder is the first piece of type $B$, we call it $B_1$.
The second piece of this type is on the bottom side of
the $x$ ear, connected to a point $o_2'$ on the trefoil below the centre of the $x$ ear. We call it $B_2$.

\item[Type $C$] The four pieces of type $C$ are tetrahedra  used to
close the gaps between pieces of type $A$ and type $B$: %(we will describe the decomposition more precisely below):
 $pp'o_1o_2$ connects $B_1$ to $A_1$, $pp'o_2o_2'$ connects $B_1$ to $B_2$, $qq'o_1'o_2'$ connects $B_2$ to $A_2$ and $qq'o_2o_2'$ connects $B_2$ to $B_1$.

\item[Type $D$] There is only one piece of type $D$, it consists of the remainder of the complement of the trefoil in $\mathbb{S}^3$.
\end{description}

\begin{figure}[ht]
\begin{subfigure}{5cm}
                \includegraphics[width=5cm]{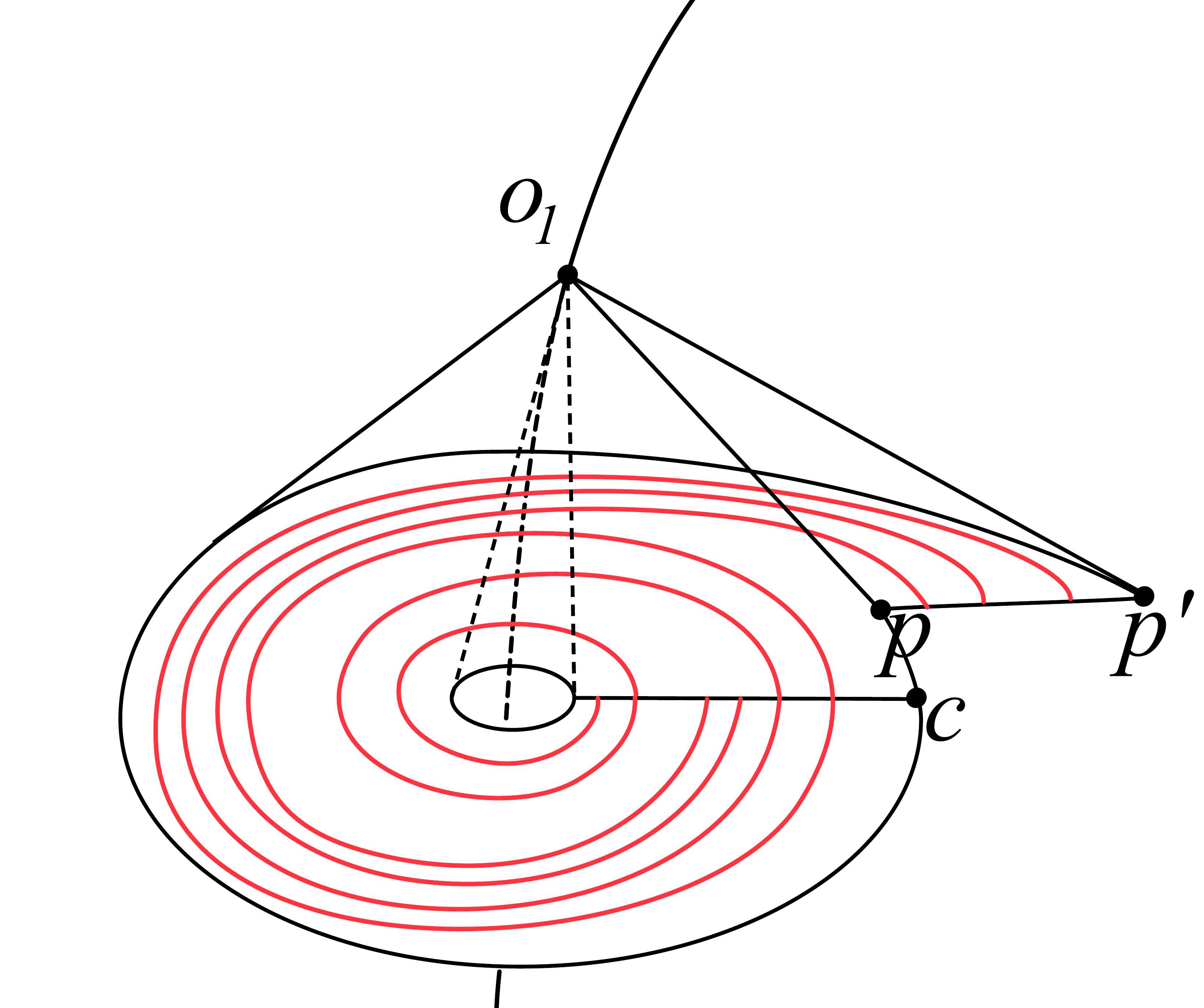}
                \caption{the $A_1$ piece}
        \end{subfigure}%
    \begin{subfigure}{5cm}
\includegraphics[width=5cm]{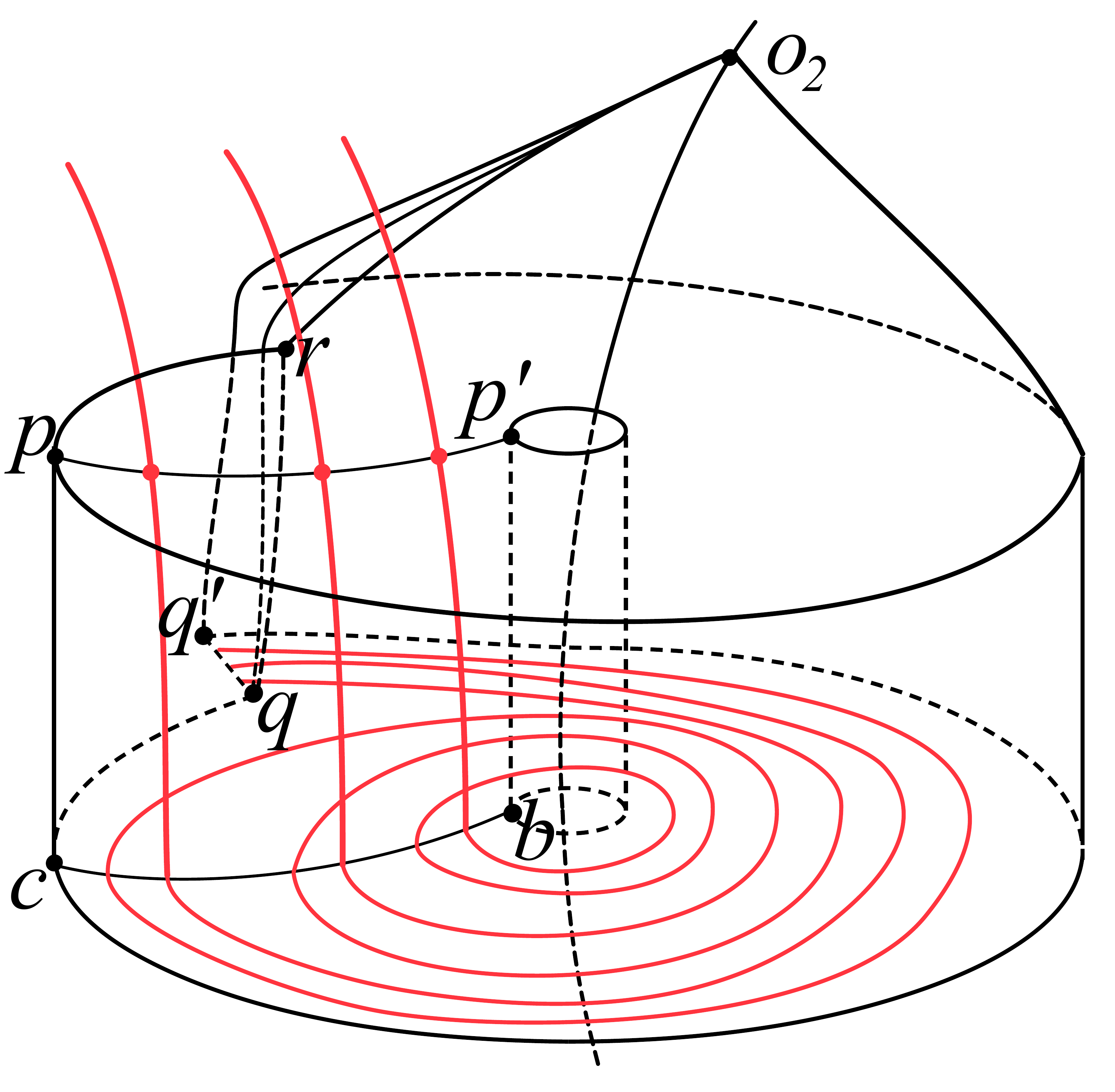}
\caption{the $B_1$ piece}
 \end{subfigure}
 \begin{subfigure}{4cm}
                \includegraphics[width=4cm]{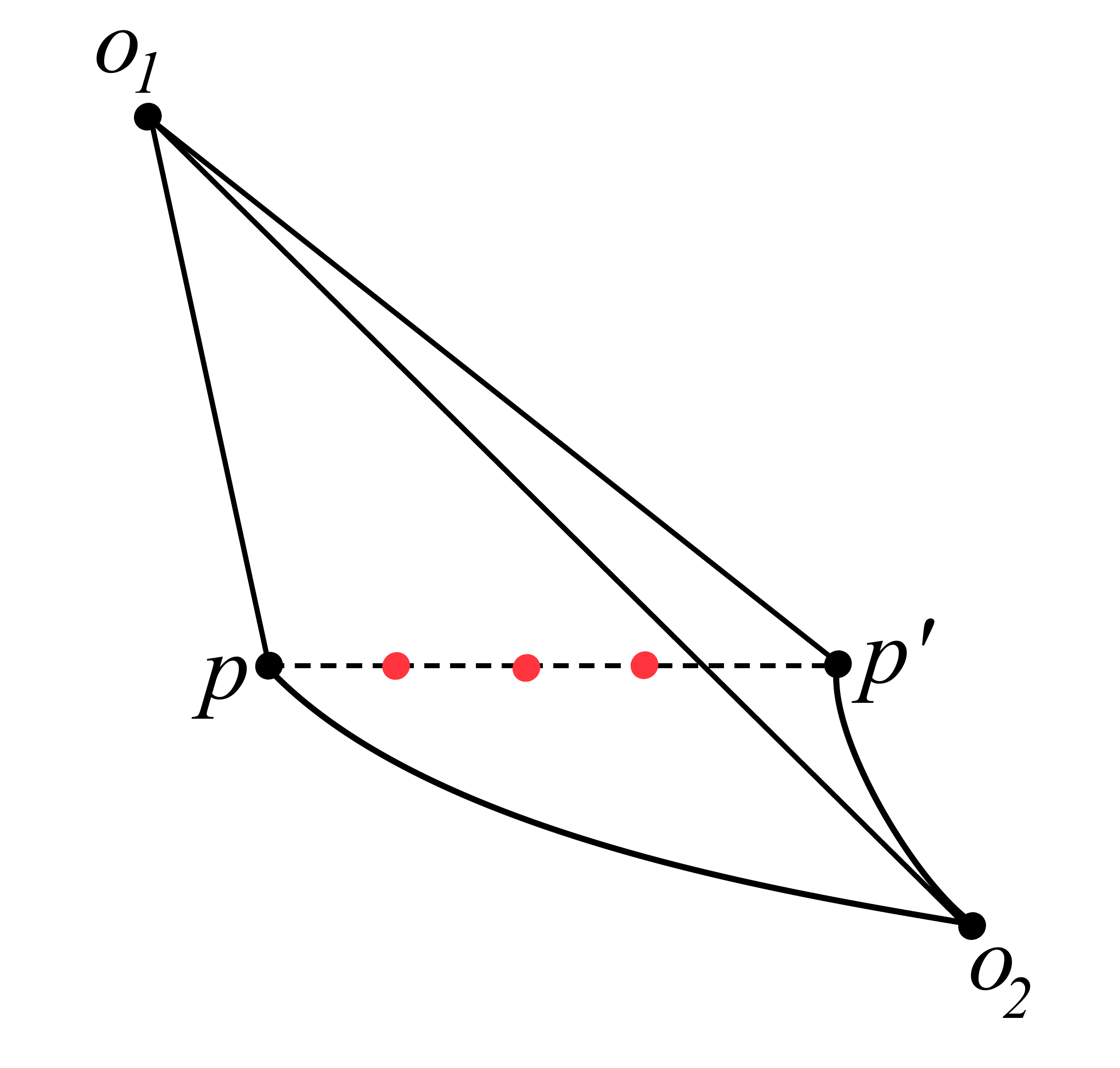}
                \caption{the $C_1$ piece}
        \end{subfigure}%
\caption{}\label{blocks}
\end{figure}

These pieces fit together as follows (as usual, relative position adjectives refer to Figure~\ref{blocks}):
\begin{enumerate}[(i)]
\item Piece $A_1$ is glued to piece $B_2$  along the $x$ ear.  It is
   glued to a $C$ piece that we will call $C_1$ along the triangle $pp'o_1$. The remainder of its coning
   face is glued to the $D$ piece.
\item Piece $B_1$ is glued to piece $A_2$ along the $y$ ear.  Recall that $B_1$ has 
   two $pp'o_2$ faces resulting from the double coning along $pp'$. It is glued to $C_1$ (connecting it to $A_1$)
   along its front $pp'o_2$ face  and it is glued to another $C$ piece which we call $C_2$ (connecting it to $B_2$) along its back $pp'o_2$ face.  It is glued to piece $B_2$ along the square $pcqr$
   and it is glued to yet another $C$ piece which we call $C_3$ along $qq'o_2$ (connecting it to $B_2$). 
   The remaining outer faces of piece $B_1$ are glued to the $D$ piece.
\item The situation for piece $A_2$ (respectively piece $B_2$) is analogous to that of piece $A_1$ (respectively piece $B_1$).
\item Each $C$ piece is connected to either $A_i$ and $B_i$ or $B_i$ and
   $B_j$. It is glued to the $D$ piece  along its two faces that do not intersect $\gamma$.
\end{enumerate}

\subsubsection{Pre-triangulation}\label{pre-triangulation}
Our next step is to triangulate the exterior faces of the pieces. This ``pre-triangulation" 
will then be completed to a triangulation of the interiors.  Gluing the
pieces along triangles will ensure the compatibility of our volume estimates.

We start with  piece $A_1$. There, we deem the points $o_1$, $p$, $p'$, $a$, $c$ as well as all intersections of $pp'$ and $ac$ with $\gamma$ to be vertices. After adding the edges indicated in Figure~\ref{blocks}, % ($a$ is connected to itself, $a$ is connected to $o_1$, $p$ is connected to $o_1$, $p'$ is connected to $o_1$, $p'$ is connected to $c$ and $c$ is connected to $p$)
  our next set of edges consists  of  all resulting subsegments of  $ac$ and  $pp'$. The vertices along $pp'$ corresponding to intersections with $\gamma$ are then all  joined by edges to  $o_1$.
This turns the face $pp'o_1$ on which $A_1$ is glued to $C_1$ 
into a union of triangles. 
Moreover, the faces contained in the $x$ ear of the template band
along which piece $A_1$ is glued to  piece $B_2$ 
turn into either bigons, triangles or higher valence polygons as shown in Figure~\ref{fig:band}. 
%These polygons occur when there are two or more strands coming in from the $y$ ear. 
In the latter cases, we fix any triangulation of the polygon within the
template ear.  In the case of a bigon, 
the two parallel edges between two $\gamma$ segments on a face of 
$A_1$ can be collapsed along the $\gamma$ segments in the direction of the template semi-flow,
resulting in a single edge along $pp'$. The exterior faces of piece $A_2$ are triangulated in an analogous manner.
%We slid the edge that is further along the gamma segments (in the direction of the semi flow) upwards (in the direction opposite to the semi flow)
%so that the edges on the segment $pp'$ are unchanged.

\begin{figure}[ht]
\includegraphics[width=5.5cm]{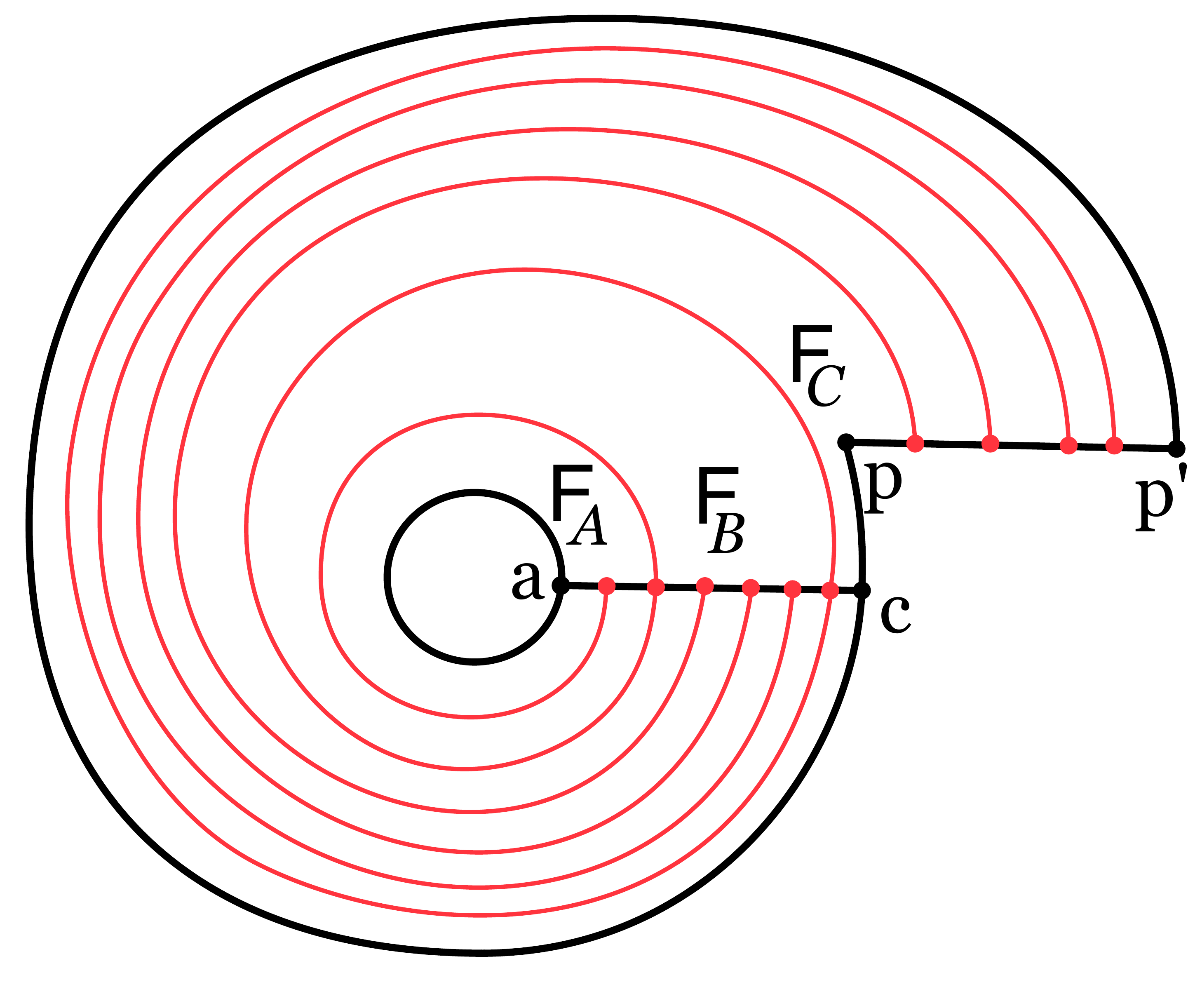}
\caption{Different polygons  induced on the $x$ ear of the  template band. Here, $F_A$ is a triangle, $F_B$ is a pentagon that would be subdivided into $3$ triangles, $F_C$ is a square and the other faces are all bigons which would collapsed into single edges along $pp'$.}\label{fig:band}
\end{figure}

To treat  piece $B_1$, observe that its face contained in the $y$ ear of the template  is already
triangulated by the choices and identifications made for the corresponding face of piece $A_2$. %(with the identification of  any two parallel edges to one edge as above).
After adding the edges indicated in Figure~\ref{blocks}, to triangulate its remaining faces,  we first add edges  connecting all intersection points of $\gamma$ with $qq'$ to $o_2$  and then add edges connecting  all intersection points
of $\gamma$ with $pp'$ to $o_2$ in two different ways -- one in each of the two $pp'o_2$ faces  on each side of the strands coming in from the $x$-band of the template. Finally, we add a diagonal edge $pq$ to the square face $prqc$. The exterior faces of piece $B_2$ are triangulated in an  analogous manner.
%see Figure~\ref{blocks}.

To triangulate the exterior faces of pieces of type $C$, we simply connect each ideal vertex to the points $o_1$ and $o_2$ or $o_1'$ and $o_2'$.

%Thus the the
%pre-triangulations of the $A$ and $B$ pieces can be made into
%triangulations in a way that is consistent with the decomposition.
%
%It remains to complete the triangulation on pieces of type $C$.  For this
%where we connect the $\gamma$ intersections on $pp'$ with the entire edge
%$o_1o_2$ and not just the two points.

\subsection{A bound for the modular surface}\label{subsec:piece-volume-bounds}
Recall that there is a uniform upper bound for the volume of geodesic
hyperbolic tetrahedra.  In particular, we can bound the volume of any piece
by the number of tetrahedra it contains.  In this section, we use this observation to establish Theorem~\ref{thm:mainthm} for the modular surface in two steps. First, we bound the number of tetrahedra in pieces of type $A,C$ and $D$ by $n_\gamma$. This
is sufficient for our purposes since, by Lemma \ref{lem:length}, upper bounds linear
in $n_\gamma$ are also linear in $\abs{\gamma}$. Unfortunately, our triangulation of pieces of type $B$ will require a  number of tetrahedra that is quadratic in $n_\gamma$. Accordingly, our second step  uses the Lobachevsky function to show that we still obtain a linear bound in $|\gamma|$.

\begin{prop}\label{prop:A-volume}
The volume of a piece of type $A$, $C$ or $D$ is at most linear in $n_\gamma$ and, thus, is at most linear in $|\gamma|$.
\end{prop}
\begin{proof}

To begin, we complete our pre-triangulations of these pieces to triangulations of their interior. 
Since the piece of type $D$ is a fixed polyhedron which doesn't intersect $\gamma$,  we can endow it with an arbitrary fixed triangulation.
To complete the pre-triangulation  of pieces of type $C$,   
we add two edges joining each  intersection of $\gamma$ along $pp'$ (respectively  $qq'$) to the other two vertices. %piece that does no meet $\gamma$, adding a face in the interior of the piece 
%for each $\gamma$ intersection along $pp'$ (respectively $qq'$).
 This allows us to triangulate them 
 using $n_\gamma+1$ tetrahedra.  The volume
contributed by the $C$ and $D$ pieces is therefore at most linear in $\abs{\gamma}$.

Unlike these first two cases, pieces of type $A$ require a more careful count.
Recall that a bigon in our pre-triangulation corresponded to a rectangle in the $x$ (respectively $y$) ear of the template, bounded by two adjacent $\gamma$ arcs and two ``parallel'' edges (see Figure \ref{fig:band}). These bigons were eliminated by collapsing the parallel edges  along the $\gamma$ arcs in the direction of the template's semi-flow, resulting in a single edge along $pp'$ (respectively $qq'$).  If this edge happened to bound a further bigon, it would once again be collapsed along adjacent $\gamma$ arcs in the direction of the semi-flow. This process continued until the two $\gamma$ arcs emanating from the collapsed edge ceased to be adjacent, i.e., their next intersection with the branch line occurred at two points separated by further intersection with $\gamma$ coming from arcs in the $y$ (respectively $x$) ear (see
Figure \ref{Ghys's template}).
If $n$ arcs coming from the $y$ (respectively $x$) ear fall in between the endpoints of two adjacent $\gamma$ arcs, we have an $n+2$-gon in the template 
which was triangulated using $n$ triangles. We can then cone them off to $o_1$ (respectively $o_1'$) using 
$n$ tetrahedra. Since there are a total of  $n_\gamma$ arcs coming in from the $y$ (respectively $x$) ear,
the number of tetrahedra used to complete our triangulation of a piece of type $A$  is at most linear in $n_\gamma$. This completes the proof of the Proposition.
\end{proof}

For our next gain, let 
$$\Lambda(\theta) = -\int_0^\theta \log\abs{2\sin u}\mathrm{d}u.$$
be the Lobachevsky function.

\begin{lem} If $T$ is a tetrahedron with dihedral angle $\alpha$
along the edge $e$,  then
$$\Vol(T) \leq 2\Lambda\left(\frac{\alpha}{2}\right)\,.$$
\end{lem}
\begin{proof}
We may assume $T$ is an ideal tetrahedron, and let $\beta,\gamma$ be the
dihedral angles at the two other edges incident to a vertex at the end of $e$.
As recounted in \cite[Lemma 2]{Milnor:HyperbolicGeometry}, it was shown
by Lobachevsky that $\alpha + \beta + \gamma = \pi$ and that
$$\Vol(T) = \Lambda(\alpha) + \Lambda(\beta) + \Lambda(\gamma)\,.$$

It is verified in \cite[Prop.\ 6.6]{Gueritaud:CanonicalTriangulations}
that the expression on the right is a concave function on the simplex
$\alpha+\beta+\gamma=\pi$.  In particular, 
$$\Vol(T) \leq \Lambda(\alpha) + 2\Lambda\left(\frac{\pi-\alpha}{2}\right)\,.$$

By the identity (c.f. \cite[Lemma 1]{Milnor:HyperbolicGeometry})
$$\Lambda(2\theta) =
   2\Lambda(\theta) + 2\Lambda\left( \frac{\pi}{2} + \theta \right) = 
   2\Lambda(\theta) - 2\Lambda\left( \frac{\pi}{2} - \theta \right)\,,$$
we see that
$2\Lambda\left(\frac{\pi}{2}-\frac{\alpha}{2}\right) =
 2\Lambda\left(\frac{\alpha}{2}\right) - \Lambda(\alpha)$
and the claim follows.
\end{proof}

\begin{prop}\label{prop:log}
Let $\{T_i\}_{i=1}^s$ be a family of $s\geq 2$ tetrahedra in $\mathbb{H}^3$,
all sharing the same edge $e$.  Then
$$\sum_i\Vol(T_i) \leq C\log(s)$$
where $C$ is a universal constant.
\end{prop}

\begin{proof}
For $s\leq 5$ the claim follows from the uniform bound on the volumes of
tetrahedra, so from now on we assume $s\geq 6$.
Let $\alpha_i$ be the dihedral angle of the tetrahedron $T_i$ at the edge $e$,
so that the total angle $\sum_i\alpha_i\leq 2\pi$.  By the Lemma, we have
$$\sum_i \Vol(T_i) \leq 2\sum_i \Lambda\left(\frac{\alpha_i}{2}\right).$$

Since $\Lambda''(\theta) = -\cot(\theta)$ is non-positive on $[0,\pi/2]$
the Lobachevsky function is concave there and it follows that
$$\sum_i \Vol(T_i) \leq 2s \Lambda\left(\frac{\sum_i \alpha_i}{2s}\right)
                   \leq 2s \Lambda\left(\frac{\pi}{s}\right),$$
where the second inequality follows from the monotonicity of $\Lambda$
in $[0,\pi/6]$ and the fact that $\sum_i \alpha_i/2s \leq \pi/6$
(here we used the assumption $s\geq 6$).

Finally, by the absolutely convergent series expansion
\bea
\Lambda(\theta)=\theta\left(1-\log\abs{2\theta}+\sum_{k=1}^\infty \frac{B_k}{2k}\frac{(2\theta)^{2k}}{(2k+1)!}\right)
\eea
(c.f. \cite[p.\ 18]{Milnor:HyperbolicGeometry}) we have
\bea
x\Lambda(\frac{\pi}{x}) \leq c + \pi\log(x) \leq C\log(s)
\eea
for some constants $c,C>0$.
\end{proof}
We are now ready to prove:
%:Modular bound
\begin{thm}[Theorem~\ref{thm:mainthm} for the modular surface]
\label{thm:firstbound}
There is a constant $C>0$ such that, for any finite collection $\gamma$
of periodic geodesics on $\MS$, we have
$$\Vol(T^1\MS\setminus\gamma) \leq C\,\abs{\gamma}.$$
\end{thm}

\begin{proof}
In light of Equation \ref{decomposition equation} and
Proposition \ref{prop:A-volume} it remains to consider the contribution
of the pieces of type $B$ and, by symmetry, it suffices to consider 
piece $B_1$.

Before completing our pre-triangulation of this piece into a triangulation, we introduce some notation and describe additional vertices within the piece's interior.  Choose a line $l$ parallel to the
segment $pp'$ and slightly below it within the $pp'bc$ rectangle. There are $n_\gamma$ intersections of $\gamma$ with $l$, they correspond to the strands arriving from the $x$ ear into the $y$ ear. We label them from right to left as $p_1^l,p_2^l,\dots,p_{n_\gamma}^l$ and add a vertex $x_i$ between $p_i^l$ and $p_{i+1}^l$. %Note that since $p_1^l$ corresponds to the $\gamma$ strand entering the $y$ ear furthest to the right (closest to $p'$), it continues to a strand of $\gamma$ that winds the largest number of times around the $y$ ear before returning to the $x$ ear. 
 %Thus, this strand of $\gamma$ passing through point $p_1^l$ corresponds to the power $\max\{m_1,\dots,m_k\}$ in the coding of $\gamma$ (see Section \ref{sec:coding}). 
 Next, we denote the intersection point of the strand starting at $p_i^l$ with the branch-line by $p_i$ and note that these two points correspond to the same point at infinity in our $3$-manifold $T^1\MS\setminus\gamma$. %Next, add a (non ideal) vertex $x_i$
%between each of the points $p_i$ and $p_{i+1}$ of the intersection of
%$\gamma$ with $l$. Add any arc in the complement of the $x_i$'s and
%the $p_i$'s in $l$ to the set of edges.
Lastly, we denote by $[p_i,p_{i+1}]$ (respectively $(p_i,p_{i+1})$) the closed (respectively open) subinterval of the branch-line between $p_i$ and $p_{i+1}$.

We now describe the edges of our triangulation.  First, add  any arc in the complement of the $x_i$'s and $p_i^l$'s within $l$ to our set of edges.
Second, whenever there is a $\gamma$ segment starting in $[p_i,p_{i+1}]$
and reaching the interval $(p_j,p_{j+1})$ after winding around the $y$ ear, we add an edge connecting $x_i$ to $x_j$, running roughly
above this $\gamma$ segment parallel to the template and within the $B$ piece.
If there is a segment of $\gamma$ emanating from $[p_i,p_{i+1}]$
and passing through $qq'$ before returning to the branch-line, we connect $x_i$ to
$o_2$ by a ``diagonal" edge above the previously added edges  within the $B$ piece. Note that some of  these edges may correspond to many $\gamma$ segments in the $y$ ear of
the template (the bottom face of the $B$ piece) and that some of these
edges may connect a point $x_i$ to itself. 
Finally, we join $x_i$ to all $\gamma$ intersections within 
$[p_i,p_{i+1}]$ by an edge and connect all $\gamma$ intersections in $[c,p_{n_\gamma}]$ to
$o_2$ by an edge.  
We are now ready to formalize the remainder of our triangulation.

 Consider adjacent intersection points of $\gamma$ with an interval $[p_{i},p_{i+1}]$. If the corresponding strands of $\gamma$ split into different segments
$[p_j,p_{j+1}]$ and $[p_k,p_{k+1}]$ after winding around the $y$ ear, there is a $(k-j)$-gon on the template band
which we have already triangulated (see Section \ref{pre-triangulation}).
 We also have a $2(k-j)+1\,$-gon
in the plane cutting through the $x_i$'s above the template band and vertical edges connecting the two polygons as illustrated in Figure~\ref{SplittingTriangulations}.
 %An example of the polyhedron and its image when $\gamma$ segments are  taken to points  at infinity is shown in
%Figure~\ref{SplittingTriangulations}(b).  
This polyhedron is easily  triangulated using a number of tetrahedra that is linear in $k-j$.
Since there are $n_\gamma$ intervals $[p_i, p_{i+1}]$, the number of tetrahedra resulting from all such splittings is linear in $n_\gamma$.

\begin{figure}[ht]
\begin{subfigure}{5cm}
                \includegraphics[width=5cm]{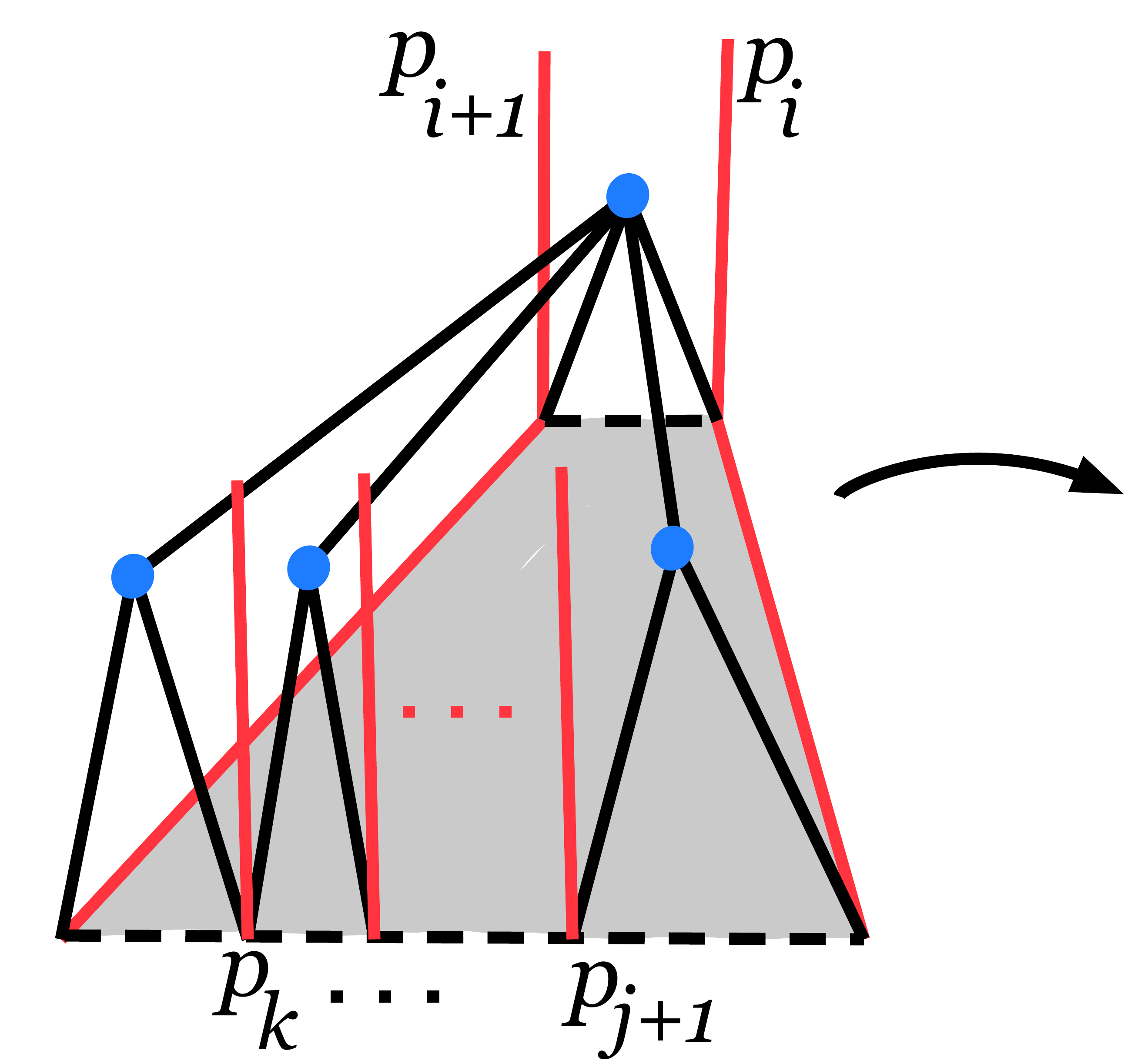}
        \end{subfigure}%
    \begin{subfigure}{5cm}
\includegraphics[width=5cm]{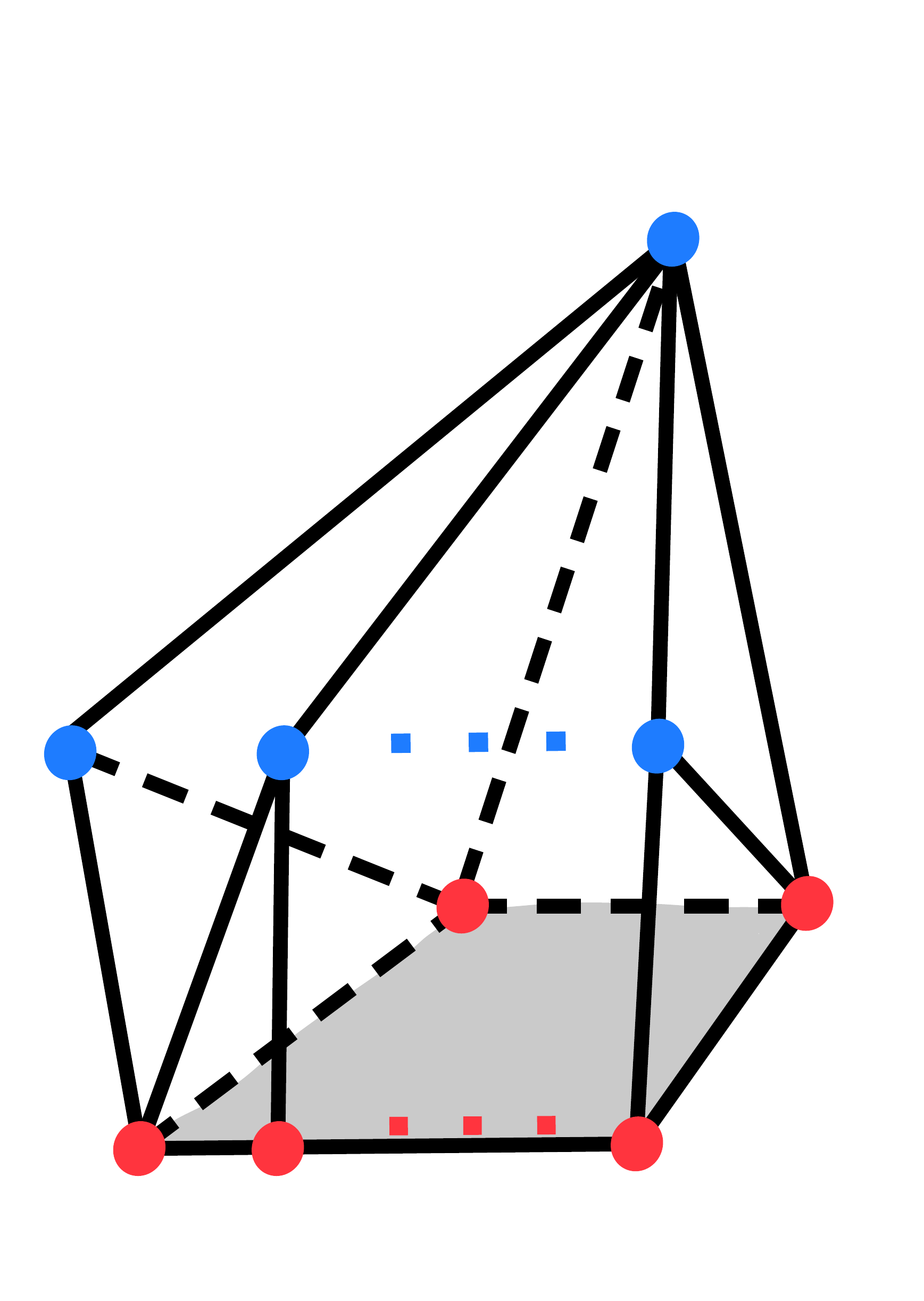}
 \end{subfigure}
\caption{The figure on the left shows the polyhedron between two segments of $\gamma$ that split before returning to the branch-line. The grey portion of the polyhedron is the previously triangulated polygon lying within the template's ear. The figure on the right
 shows the same polyhedron after  taking $\gamma$ to infinity.}\label{SplittingTriangulations}
\end{figure}

On the other hand, if the corresponding  strands of $\gamma$ continue to two adjacent intersection points in some
interval $[p_j,p_{j+1}]$ where $j\geq i$, we already have a tetrahedron
in the part of the $B$ piece above them: the convex hull of the edge
between them (this edge is unique -- both intersections with the
branch-line have been identified) and the edge $[x_i,x_j]$.
As previously noted, there are too many such tetrahedra to directly yield an adequate  volume bound.

A strand of $\gamma$ passing through $p_i$ upon entering  the $B$ piece from the $x$ ear corresponds to a power of $y$ in the coding of $\gamma$. This power determines how many times the strand will wind around the $y$ ear before exiting the $B$ piece. If we denote this power of $y$ by $m_{j_i}$, this strand will wind around the $y$ ear as a sequence
of $m_{j_i}$ segments of $\gamma$ separated by intersections with the branch-line. 
 We denote these $\gamma$ segments  by $\gamma_1,\dots,\gamma_{j_i}$ and observe that  
each of them is adjacent to at most two rectangles within the template. Indeed, a segment is  adjacent to a rectangle whenever it does
not split from the $\gamma$ arc crossing the branch-line next to it on the left or the right before returning to the branch-line. In such a situation, starting at $p_i$ and on either side of the $\gamma$ segments, there may be a chain of consecutive rectangles separated by the branch-line and adjacent to the segments $\gamma_1$, $\gamma_2$,$\ldots \gamma_{j_k}$.
This chain terminates just before  a splitting of  segment $\gamma_{j_k}$ and its neighbouring $\gamma$ arc if there is one, or at the intersection with $qq'$.
Thus, on each side, we have a chain of
at most $m_{j_i}$ rectangles and above each rectangle we have a tetrahedron
(as indicated in the previous paragraph).  Recalling that the branch-line edges of these rectangles are all identified as a single edge (see Section \ref{pre-triangulation}), the rectangles in the chain all
have this edge in common. Therefore, by Proposition~\ref{prop:log}, the tetrahedra belonging to this chain 
contribute at most $C\log(m_{j_i})$ to the volume of $T^1\MS\setminus\gamma$. 

Note that once $\gamma_{j_k}$ does split from a neighbouring $\gamma$ arc, then the region above the splitting was already triangulated (see Figure \ref{SplittingTriangulations}). Moreover, the regions in the template that start on the branch-line between the endpoints of the splitting segments were also previously taken care of as they are adjacent to $\gamma$ segments that start at points $p_j,\dots,p_k$ that lie between these endpoints. As such, it only remains to triangulate the region above the plane section parallel to the $y$ ear containing the $x_i$'s which are not connected to $o_2$ by an edge. 

Observe that this plane section is divided into a number of  triangles and squares that is linear in $n_\gamma$. Upon subdividing the squares into triangles, we obtain a planar region subdivided into a number of  triangles that is also linear in $n_\gamma$. We can then triangulate the upper region 
 by coning these triangles to $o_2$. %The number of new tetrahedra in this part will then be linear in $n_\gamma$ as well.
To conclude, the total volume of the $B$ piece is thus given by
$$ A(n_{\gamma}+\sum_{i=1}^{n_\gamma} \log(k_i))$$
for an appropriate constant $A>0$. This bound  is at most linear in
$\abs{\gamma}$ by Lemma~\ref{lem:length}.
\end{proof}

%%%%%%%%%%%%%%%%%%%%%%%%%%%%%%%%%%%%%%%%%%%%%%%%%%%%%%%%%%%%
\section{An upper bound for arbitrary hyperbolic surfaces}\label{sec:upperbound}
Having established Theorem \ref{thm:mainthm} for the modular surface, we now obtain its general form in three stages.  First, we show that the claim is topological, in that
it is independent of the choice of hyperbolic structure on a surface:
\begin{lem} Let $g$ and $h$ be two hyperbolic structures on the surface $S$.
If Theorem~\ref{thm:mainthm} holds for $(S,g)$, then it also holds for
$(S,h)$.
\end{lem}
\begin{proof}
This follows immediately from the fact that the metrics
$g$ and $h$ on $S$ are bi-Lipschitz.
Indeed, any Dirichlet fundamental domain is the union of a set of cusp
neighbourhoods and a compact complement, so we can take any diffeomorphism
which is an isometry near the cusps.  Furthermore, we can take the
diffeomorphism to be homotopic to the identity.  In that case, the
diffeomorphism will preserve the isotopy classes of curves.  Since every
essential closed geodesic is the shortest curve in its isotopy class,
it follows that the ratios of lengths of the geodesics corresponding
to each such isotopy class in the two hyperbolic structures are uniformly
bounded.
\end{proof}
\begin{rem}
The supremum of those ratios is exactly the infimum
of the Lipschitz constants of maps between the structures;
see \cite[Thm.\ 8.5]{thurston1998minimal}.  Control of the ratios also
follows from the bounds in \cite{lenzhen2012bounded}.
	
\end{rem}

Next, we show that the claim passes to covers:

\begin{prop} Let $\hat{S}\to S$ be a degree $d$ regular covering
of hyperbolic Riemann surfaces.  Suppose Theorem~\ref{thm:mainthm} holds
for $S$ with constant $A_S$.  Then it also holds for $\hat{S}$ with the
constant $d\cdot A_S$.
\end{prop}
\begin{proof}
Given a filling set $\gamma$ of periodic geodesics on $\hat{S}$, let
$\gamma'$ be the union of all translates of $\gamma$ under the group
of deck transformations and let $\delta$ be the projection of $\gamma$
(and $\gamma'$) to $S$.  Then, for each component (simple closed geodesic)
$\delta_i$ in $\delta$, $\gamma$ must contain one of the lifts of $\delta_i$
to $\hat{S}$, which are themselves closed geodesics of length at least
$l_S(\delta_i)$.  It follows that $l_S(\delta) \leq l_{\hat{S}}(\gamma)$
and, hence, that

\begin{align*}
\Vol(T^1\hat S\setminus\gamma)
 & \leq \Vol(T^1\hat S\setminus\gamma') = d\cdot \Vol(T^1S\setminus\delta) \\
 & \leq d\cdot A_S\, l_S(\delta) \\
 & \leq \left(d\cdot A_S\right) l_{\hat{S}}(\gamma).
\end{align*}
\end{proof}

\begin{figure}[ht]
\includegraphics[width=12cm]{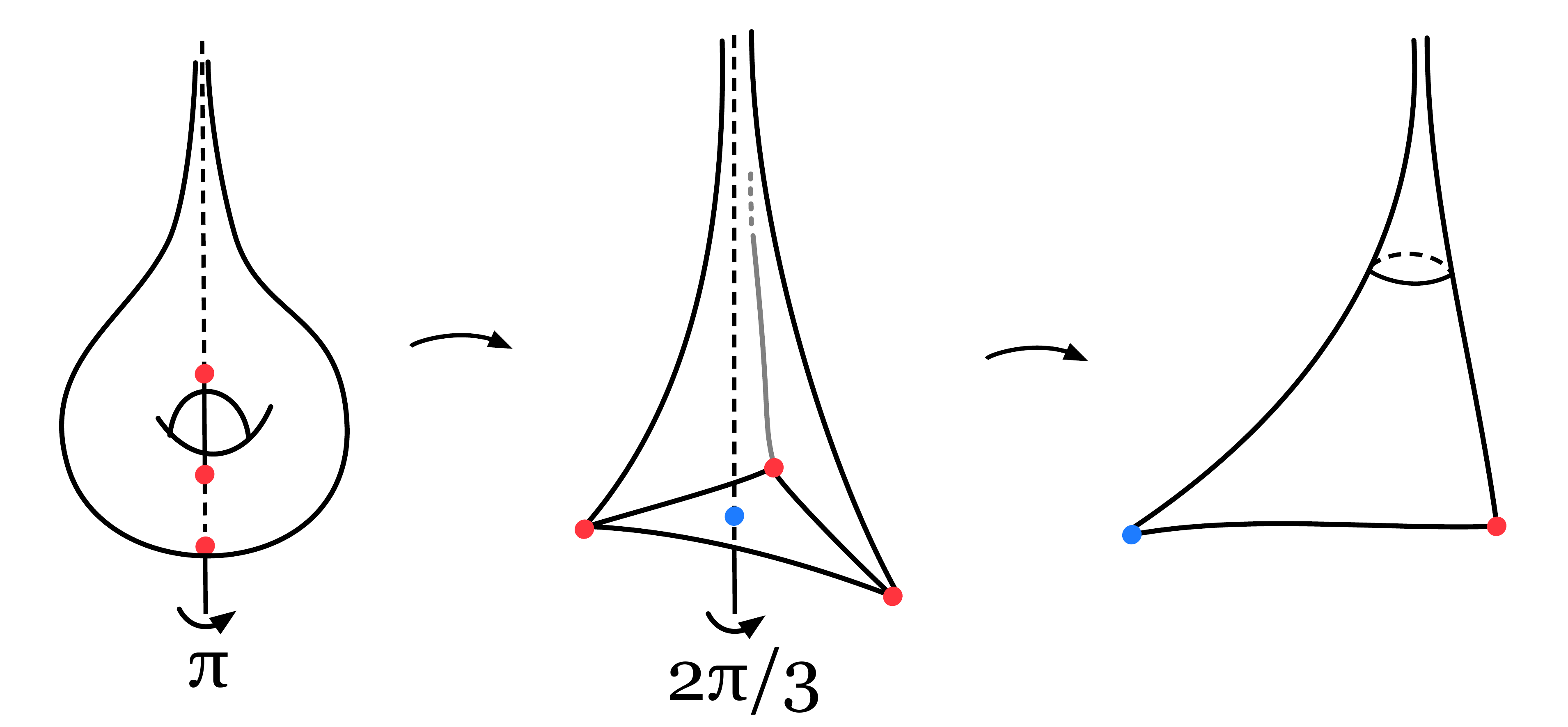}
\caption{The cover of $\MS$ by a once punctured torus.}\label{cover}
\end{figure}
\begin{figure}[ht]
\includegraphics[width=7cm]{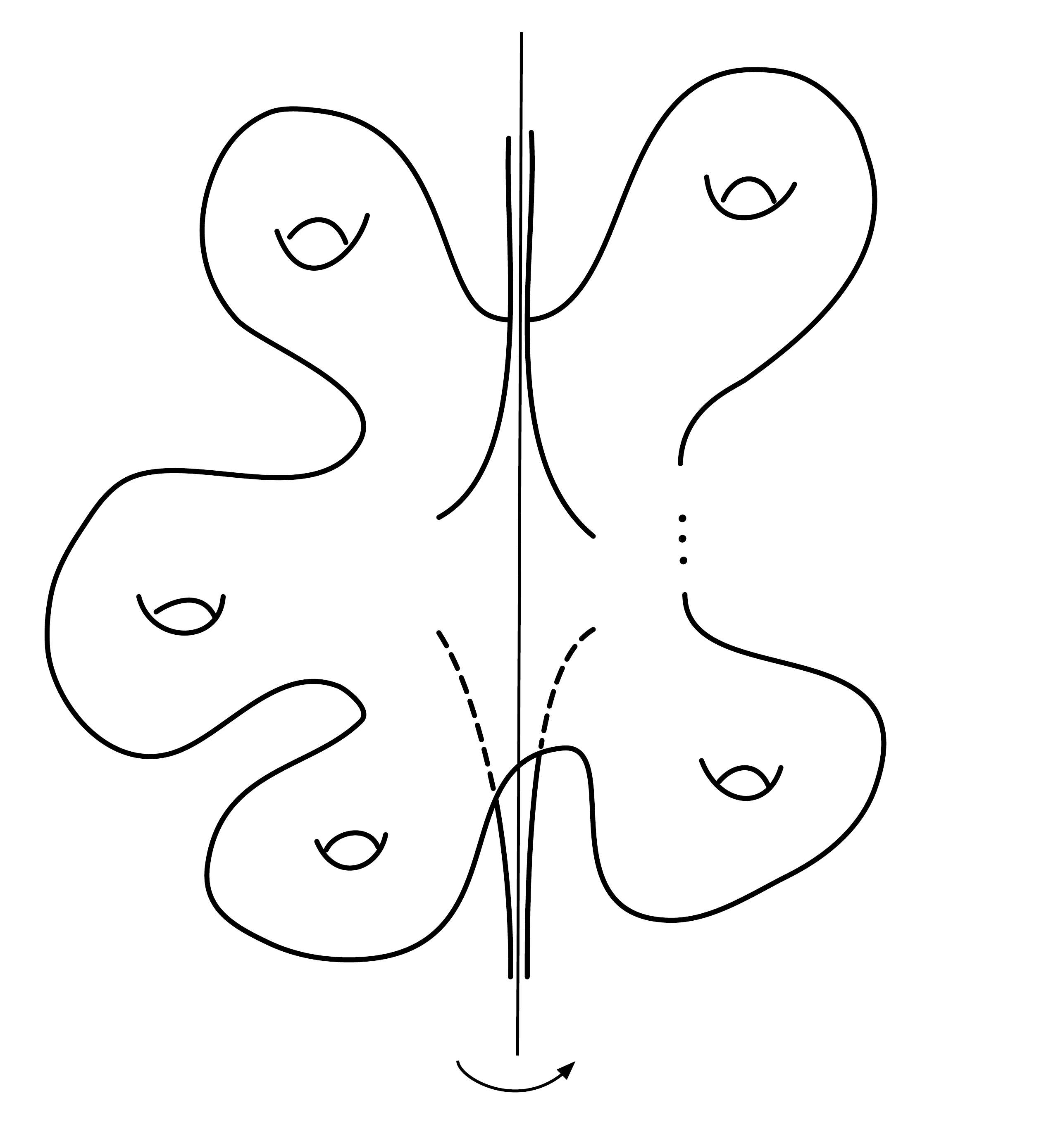}
\caption{The cover of the twice punctured torus by a twice punctured genus $g$ surface.}\label{genus}
\end{figure}
Finally, we use ideas of
Brooks on adding punctures to hyperbolic surfaces to deduce the general case:
\begin{proof}[Proof of Theorem~\ref{thm:mainthm}]
We have seen that the Theorem holds for any surface which is a cover
of the modular surface.
This includes the once-punctured torus (see Figure \ref{cover}) and its double
cover, the twice-punctured torus. (To construct this double cover,
arrange the two punctures symmetrically about the axis of revolution of the
torus and apply a rotation by $\pi$.)
In the same manner, there is a $6k$-fold covering map from a $k$-punctured
torus to $\MS$. 
Next, a rotation through the two cusps in a symmetrically arranged
twice-punctured surface of genus $g$ (see Figure \ref{genus}) yields
a cover of the twice punctured torus, and thus of $\MS$.
The same construction gives a covering map from $S_{g,2+kg}$
(the surface of genus $g$ with $2+kg$ punctures) to $\MS$ by
permuting the punctures and mapping to the $k+2$ punctured torus.
Lastly, recall that the thrice punctured sphere is a $6$-fold cover of the
modular surface.  The same strategy now produces a $6(n-3)$-fold covering
map from a punctured sphere $S_{0,n}$ with $n>3$ to $\MS$ by arranging
an axis through two of the punctures and using a rotation acting
transitively on the rest of the punctures. 

We are now ready to consider a collection of periodic geodesics $\gamma$
on an arbitrary hyperbolic surface $S_{g,n}$.  Since we have already dealt
with the case where $n\equiv 2$ mod $g$, we just need a way
to increase the number of punctures by $k$ (chosen such that
$n+k\equiv 2$ mod $g$) to obtain a hyperbolic surface $S_{g,n+k}$ which
is a topological cover of $\MS$.  For this, mark $k$ arbitrary
points on $S_{g,n}$.  By \cite{Brooks:PlatonicSurfaces} (see especially
Lemma 3.1), there is a choice of metrics on $S_{g,n}$ and $S_{g,n+k}$
such that every closed geodesic in $S_{g,n}$ can be isotoped away from
the punctures with the resulting closed geodesic in $S_{g,n+k}$ having
length bounded by a constant times its length in $S_{g,n}$.

Accordingly, given a family $\gamma$ of closed geodesics in $S_{g,n}$,
let $\gamma'$ be the isotopic family of curves lying off the cusps.
We then have
\begin{align*}
\Vol(T^1S_{g,n}\setminus{\gamma}) & = \Vol(T^1S_{g,n}\setminus{\gamma'}) \\
 & \leq \Vol(T^1S_{g,n+k}\setminus{\gamma'}) \\
 & \leq C \abs{\gamma'} \leq C' \abs{\gamma}\,.
\end{align*}
Here, the first equality is the topological invariance of the volume
and the following inequality is the monotonicity of the volume of hyperbolic
three-manifolds under drilling (the circles lying over the
punctures).  The third claim is the Theorem in the case of $S_{g,n+k}$,
and the last step is the linear relation between the length of $\gamma'$
in $S_{g,n+k}$ and the length of $\gamma$ in $S_{g,n}$.
\end{proof}

\bibliographystyle{amsplain}
\bibliography{volumes}

\providecommand{\bysame}{\leavevmode\hbox to3em{\hrulefill}\thinspace}
\providecommand{\MR}{\relax\ifhmode\unskip\space\fi MR }
\providecommand{\MRhref}[2]{%
  \href{http://www.ams.org/mathscinet-getitem?mr=#1}{#2}
}
\providecommand{\href}[2]{#2}
\begin{thebibliography}{10}

\bibitem{adams1985thrice}
Colin~C. Adams, \emph{Thrice-punctured spheres in hyperbolic {$3$}-manifolds},
  Trans. Amer. Math. Soc. \textbf{287} (1985), no.~2, 645--656. \MR{768730
  (86k:57008)}

\bibitem{Anosov:GeodesicFlows}
Dimitri~V. Anosov, \emph{Geodesic flows on closed {R}iemann manifolds with
  negative curvature.}, Proceedings of the Steklov Institute of Mathematics,
  No. 90 (1967). Translated from the Russian by S. Feder, American Mathematical
  Society, Providence, R.I., 1969. \MR{0242194 (39 \#3527)}

\bibitem{Artin:SymbolicCoding}
Emil Artin, \emph{Ein mechanisches system mit quasiergodischen bahnen}, Abh.
  Math. Sem. Univ. Hamburg \textbf{3} (1924), no.~1, 170--175. \MR{3069425}

\bibitem{BirmanWilliams:KnottedOrbitsII}
Joan~S. Birman and R.~F. Williams, \emph{Knotted periodic orbits in dynamical
  system. {II}. {K}not holders for fibered knots}, Low-dimensional topology
  ({S}an {F}rancisco, {C}alif., 1981), Contemp. Math., vol.~20, Amer. Math.
  Soc., Providence, RI, 1983, pp.~1--60. \MR{718132 (86a:58084)}

\bibitem{BrandtsPinskySilberman:KnotVolumesNumerics_preprint}
Alex Brandts, Tali Pinsky, and Lior Silberman, \emph{Volumes of hyperbolic
  three-manifolds associated to modular links}, In preparation.

\bibitem{bridgeman1998bounds}
Martin Bridgeman, \emph{Bounds on volume increase under {D}ehn drilling
  operations}, Proc. London Math. Soc. (3) \textbf{77} (1998), no.~2, 415--436.
  \MR{1635161 (99f:57019)}

\bibitem{Brooks:PlatonicSurfaces}
Robert Brooks, \emph{Platonic surfaces}, Comment. Math. Helv. \textbf{74}
  (1999), no.~1, 156--170. \MR{1677565 (99k:58185)}

\bibitem{dehornoy2015geodesic}
Pierre Dehornoy, \emph{Geodesic flow, left-handedness and templates}, Algebr.
  Geom. Topol. \textbf{15} (2015), no.~3, 1525--1597. \MR{3361144}

\bibitem{FoulonHasselblatt:ContactAnosov}
Patrick Foulon and Boris Hasselblatt, \emph{Contact {A}nosov flows on
  hyperbolic 3-manifolds}, Geom. Topol. \textbf{17} (2013), no.~2, 1225--1252.
  \MR{3070525}

\bibitem{Fried:TransitiveAnosov}
David Fried, \emph{Transitive {A}nosov flows and pseudo-{A}nosov maps},
  Topology \textbf{22} (1983), no.~3, 299--303. \MR{710103 (84j:58095)}

\bibitem{Ghys:KnotsDynamics}
{\'E}tienne Ghys, \emph{Knots and dynamics}, International {C}ongress of
  {M}athematicians. {V}ol. {I}, Eur. Math. Soc., Z\"urich, 2007, pp.~247--277.
  \MR{2334193 (2008k:37001)}

\bibitem{Gueritaud:CanonicalTriangulations}
Fran{\c{c}}ois Gu{\'e}ritaud, \emph{On canonical triangulations of
  once-punctured torus bundles and two-bridge link complements}, Geom. Topol.
  \textbf{10} (2006), 1239--1284, With an appendix by David Futer. \MR{2255497
  (2007g:57009)}

\bibitem{kelmer2012quadratic}
Dubi Kelmer, \emph{Quadratic irrationals and linking numbers of modular knots},
  J. Mod. Dyn. \textbf{6} (2012), no.~4, 539--561. \MR{3008409}

\bibitem{lenzhen2012bounded}
Anna Lenzhen, Kasra Rafi, and Jing Tao, \emph{Bounded combinatorics and the
  {L}ipschitz metric on {T}eichm\"uller space}, Geom. Dedicata \textbf{159}
  (2012), 353--371. \MR{2944537}

\bibitem{Milnor:HyperbolicGeometry}
John Milnor, \emph{Hyperbolic geometry: the first 150 years}, Bull. Amer. Math.
  Soc. (N.S.) \textbf{6} (1982), no.~1, 9--24. \MR{634431 (82m:57005)}

\bibitem{sarnak2010linking}
Peter Sarnak, \emph{Linking numbers of modular knots}, Commun. Math. Anal.
  \textbf{8} (2010), no.~2, 136--144. \MR{2587769}

\bibitem{Series:ModularSurfaceContinuedFractions}
Caroline Series, \emph{The modular surface and continued fractions}, J. London
  Math. Soc. (2) \textbf{31} (1985), no.~1, 69--80. \MR{810563}

\bibitem{Thurston:GeomDynSurfaceDiff}
William~P. Thurston, \emph{On the geometry and dynamics of diffeomorphisms of
  surfaces}, Bull. Amer. Math. Soc. (N.S.) \textbf{19} (1988), no.~2, 417--431.
  \MR{956596 (89k:57023)}

\bibitem{thurston1998minimal}
\bysame, \emph{Minimal stretch maps between hyperbolic surfaces}, preprint
  \texttt{arXiv:math.GT/9801039}, 1998.

\end{thebibliography}

\end{document}